\documentclass[11pt]{amsart}
\usepackage{amsmath,amssymb,amscd,amsfonts,graphicx,latexsym,epsfig,psfrag,url,color,xr,mathtools,enumitem,transparent,fullpage,hyperref,float,mathrsfs,wrapfig,rotating,array }
\usepackage[all]{xy}
\usepackage[foot]{amsaddr}

\newtheorem{theorem}{Theorem}[section]

\newtheorem{question}[theorem]{Question}

\newtheorem{proposition}[theorem]{Proposition}
\newtheorem{lemma}[theorem]{Lemma}

\theoremstyle{definition}

\theoremstyle{remark}

\theoremstyle{plain}
\newcommand{\thistheoremname}{}
\newtheorem{genericthm}[theorem]{\thistheoremname}

\newtheorem*{genericthm*}{\thistheoremname}
\newenvironment{namedthm*}[1]
  {\renewcommand{\thistheoremname}{#1}%
   \begin{genericthm*}}
  {\end{genericthm*}}
  
\newcommand\cM{\mathcal{M}}

\newcommand{\bR}{\mathbb{R}}
\newcommand{\bH}{\mathbb{H}}
\newcommand{\bC}{\mathbb{C}}
\newcommand{\bZ}{\mathbb{Z}}

\newcommand{\bRP}{\mathbb{RP}}
\newcommand{\bCP}{\mathbb{CP}}

\newcommand{\fg}{\mathfrak{g}}

\newcommand{\on}{\operatorname}

\newcommand\id{\on{id}}

\newcommand{\re}{\on{re}}
\newcommand{\im}{\on{im}}

\newcommand{\inte}{{\on{int}}}

\newcommand{\Cl}{{\on{Cl}}}
\newcommand{\reduc}{{\hspace{-0.15em}/\!\!/\hspace{-0.05em}}}
\newcommand{\AC}{{\on{AC}}}

\newcommand\qu{/\kern-.7ex/} 
\newcommand\lqu{\backslash \kern-.7ex \backslash}

\newcommand{\ol}{\overline}
\newcommand{\ul}{\underline}

\newcommand{\wt}{\widetilde}

\def\d{\delta}

\def\om{\omega}

\renewcommand{\d}{{\rm d}}

\def\dt{{\rm d}t}

\newcounter{qcounter}

\newcommand\quotient[2]{
        \mathchoice
            {
                \text{\raise1ex\hbox{$#1$}\Big/\lower1ex\hbox{$#2$}}%
            }
            {
                #1\,/\,#2
            }
            {
                #1\,/\,#2
            }
            {
                #1\,/\,#2
            }
    }

\newcommand\quoti[2]{
                \text{\raise1ex\hbox{$#1$}/\lower1ex\hbox{$\scriptstyle#2$}}
  }

\newcommand\quot[2]{
                \text{\raise1ex\hbox{$#1\!\!$}/\lower1ex\hbox{$\!\scriptstyle#2$}}
  }

\newcommand\quo[2]{
                \text{\raise.8ex\hbox{$\scriptstyle#1\!$}/\lower.8ex\hbox{$\!\scriptstyle#2$}}
  }

\newcommand\qq[2]{
                \text{\raise.8ex\hbox{$#1\!$}/\lower.8ex\hbox{$#2$}}
}

\begin{document}

\title{Explicit constructions of quilts with seam condition coming from symplectic reduction}
\author{Nathaniel Bottman}
\address{Department of Mathematics, University of Southern California, 3620 S Vermont Ave, Kaprielian Hall Rm.\ 400C, Los Angeles, CA 90089}
\email{\href{mailto:bottman@usc.edu}{bottman@usc.edu}}

\maketitle

\begin{abstract}
Associated to a symplectic quotient $M\reduc G$ is a Lagrangian correspondence $\Lambda_G$ from $M\reduc G$ to $M$.
In this note, we construct in two examples quilts with seam condition on such a correspondence, in the case of $S^1$ acting on $\bCP^2$ with symplectic quotient $\bCP^2\reduc S^1 = \bCP^1$.
First, we exhibit the moduli space of quilted strips that would, if not for figure eight bubbling, identify the Floer chain groups $CF(\gamma,S_\Cl^1)$ and $CF(\bRP^2,T_\Cl^2)$, where $\gamma$ is the connected double-cover of $\bRP^1$.
Second, we answer a question due to Akveld--Cannas~da~Silva--Wehrheim by explicitly producing a figure eight bubble which obstructs an isomorphism between two Floer chain groups.
The figure eight bubbles we construct in this paper are the first concrete examples of this phenomenon.
\end{abstract}

\section{Introduction}
\label{sec:intro}

Suppose that $L_{01}$ is a Lagrangian correspondence from $M_0$ to $M_1$, and $L_0 \subset M_0$ is a Lagrangian.
Then $\pi_1\colon M_0 \times M_0 \times M_1 \to M_1$ restricts to a Lagrangian immersion of
\begin{align}
L_0 \times_{M_1} L_{01} \coloneqq (L_0\times L_{01}) \cap (\Delta_{M_0}\times M_1)
\end{align}
into $M_1$, as long as the intersection appearing in the definition of $L_0\times_{M_1} L_{01}$ is transverse.
(In this situation, $L_0$ and $L_{01}$ are said to have \emph{immersed composition}, and the image $\pi_1(L_0\times_{M_1}L_{01})$ is denoted $L_0 \circ L_{01}$.)
If (1) $L_1 \subset M_1$ is another Lagrangian, (2) both $L_0 \circ L_{01} \subset M_0$ and $L_1 \circ L_{01}^T \subset M_1$ are embedded compositions (where $L_{01}^T$ is the result of regarding $L_{01}$ as a correspondence from $M_1$ to $M_0$), and (3) we make assumptions on the geometry to exclude all bubbling and ensure that the relevant moduli spaces are cut out transversely, then the following \emph{Wehrheim--Woodward isomorphism} of Floer cohomology groups holds (taking coefficients in $\bZ/2$, as we will do throughout this paper):
\begin{align}
\label{eq:WW_iso}
HF(L_0\circ L_{01}, L_1) \simeq HF(L_0, L_1\circ L_{01}^T).
\end{align}
When any of these assumptions are weakened, this isomorphism may not hold.
The obstruction is \emph{figure eight bubbling}, and in this paper we produce the first concrete examples of this bubbling phenomenon, in the context of symplectic reduction.

Suppose that $G$ is a compact Lie group acting in a Hamiltonian fashion on a symplectic manifold $M$, with moment map $\mu\colon M \to \fg^*$; suppose furthermore that $a \in \fg^*$ is a central element, that $\mu^{-1}(a)$ is a regular level set, and that $G$ acts freely and properly on $\mu^{-1}(a)$.
Then $M\reduc G \coloneqq \mu^{-1}(a) / G$ has the natural structure of a symplectic manifold, with $\omega_{M\reduc G}$ defined uniquely by $\pi^*\omega_{M\reduc G} = \iota^*\omega_M$.
There is a Lagrangian correspondence $\Lambda_G$ from $M\reduc G$ to $M$, defined by
\begin{align}
\Lambda_G
\coloneqq
\left\{([p],p) \:|\: p \in \mu^{-1}(a)\right\} \subset (M\reduc G)^- \times M.
\end{align}
The action we will consider in this paper is $S^1$ acting on $\bCP^2$ by
\begin{align}
e^{i\lambda}\cdot[X:Y:Z] \coloneqq [X:Y:e^{i\lambda}Z],
\end{align}
with moment map
\begin{align}
\mu\colon \bCP^2 \to \bR,
\qquad
\mu[X:Y:Z] \coloneqq -\frac 1 2\frac {|Z|^2} {|X|^2 + |Y|^2 + |Z|^2}.
\end{align}
Here and throughout the rest of this paper, we equip $\bCP^2$ with the Fubini--Study form normalized to have monotonicity constant 1, and define $\bCP^2\reduc S^1$ to be the reduction at the level set $\mu^{-1}(-\tfrac16)$.
$\bCP^2\reduc S^1$ is then $\bCP^1$ with the normalized Fubini--Study form with monotonicity constant 1.
The associated correspondence $\Lambda_{S^1}$ is diffeomorphic to $S^3$, and is therefore a monotone Lagrangian correspondence from $\bCP^1$ to $\bCP^2$.

In the upcoming two sections we explicitly construct holomorphic quilts with seam condition on $\Lambda_{S^1}$, in the following two examples:

\begin{itemize}
\item In \S\ref{sec:CP1_CP2}, we study the moduli space of rigid quilted strips mapping to $\bCP^1$ and $\bCP^2$, with boundary on the Clifford circle $S^1_\Cl$ and $\bRP^2$ and seam condition on $\Lambda_{S^1}$.
If there were no bubbling, this moduli space would produce a cobordism giving rise to an isomorphism $CF(\gamma,S^1_\Cl) \simeq CF(\bRP^2,T^2_\Cl)$ as in \eqref{eq:WW_iso}.
Such an isomorphism does not hold, as the left-hand side is a chain complex and the right-hand side is only a matrix factorization of $\id$.
We explicitly exhibit this moduli space by expressing the constituent quilts in terms of Blaschke products and Poisson integrals.
Bubbling must occur, and indeed we see four obstructing figure eight bubbles.

\medskip

\item In \S\ref{sec:question}, we answer a question posed by Akveld--Cannas da Silva--Wehrheim in a 2015 private communication \cite{acw}.
This question concerns the ``Ana Cannas'' Lagrangian $L_\AC \subset \bCP^2$ defined in \eqref{eq:L_AC}, which is Hamiltonian-isotopic to $\bRP^2$ but has the property that the composition $L_\AC \circ \Lambda_{S^1}^T = S^1_\Cl$ is an embedded Lagrangian.
Akveld--Cannas da Silva--Wehrheim observed that the lack of an isomorphism between $CF(S^1_\Cl,\bRP^1)$ and $CF(L_\AC,\bRP^1\circ\Lambda_{S^1})$ implies the existence of a figure eight bubble of the form described in Question~\ref{q:acw}, and asked whether such a bubble can be explicitly produced.
We do so in \S\ref{sec:question}.
We leave open the question of whether the figure eight bubble we produce is the unique satisfactory bubble, and also the question of how to see this quilt bubbling off from the relevant moduli space of quilted strips.
\end{itemize}

\subsection{Acknowledgements}

Mark Goresky suggested in a related situation that it could be more tractable to work with quilts whose domains are quilted half-planes rather than quilted disks; this idea turned out to be of crucial importance, and gave the author the confidence to carry out the work in this paper.
Alexandre Eremenko answered a question the author posed on MathOverflow, which was key to the construction of the quilted strips in \S\ref{sec:CP1_CP2}.
Denis Auroux made a helpful observation about the nontriviality of the double-cover $\bRP^2 \circ \Lambda_{S^1}^T \to \bRP^1$.
Katrin Wehrheim introduced the author in 2015 to the question of Akveld--Cannas~da~Silva--Wehrheim and suggested that he think about the relationship between $CF(\gamma,S^1_\Cl)$ and $CF(\bRP^2,T^2_\Cl)$.
The author thanks Paul Seidel and Chris Woodward for their encouragement in the beginning of this investigation.

The work in this paper was carried out while the author was a member at the Institute for Advanced Study and a postdoctoral researcher at Princeton University, and a visitor at the Mathematical Sciences Research Institute.
The author was supported by an NSF Mathematical Sciences Postdoctoral Research Fellowship.

\section{The moduli spaces of quilted strips relating $CF(\gamma,S^1_\Cl)$ and $CF(\bRP^2,T^2_\Cl)$}
\label{sec:CP1_CP2}

Set $\gamma \coloneqq \bRP^2 \circ \Lambda_{S^1}^T$.
Then we have $\gamma = \{[A:B:C] \in \bRP^2 \:\:|\:\: 2C^2 = A^2 + B^2\}$, so $\gamma$ is the connected double-cover of $\bRP^1 \subset \bCP^1$.
In the following lemma, we compute the generators and differentials in $CF(\gamma,S^1_\Cl)$ and $CF(\bRP^2,T^2_\Cl)$.

\begin{lemma}
\label{lem:floer_diff}
The following figure indicates the generators and differentials in $CF(\gamma,S^1_\Cl)$ and $CF(\bRP^2,T^2_\Cl)$, where each dot resp.\ arrow represents a generator resp.\ a single rigid strip.
\begin{figure}[H]
\centering
\def\svgwidth{0.55\columnwidth}
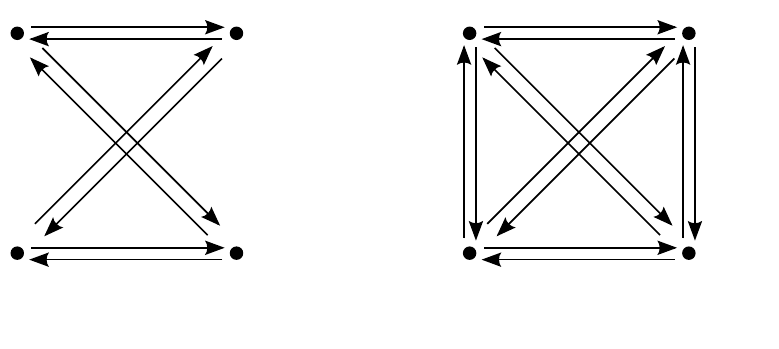
\end{figure}
\end{lemma}

\begin{proof}
We begin with $CF(\bRP^2,T^2_\Cl)$.
This chain group is generated by the points $p_{\pm\pm} \coloneqq [1:\pm1:\pm1]$, where the two signs need not be the same.
It follows from \cite[Prop.~3.1]{alston} that there are 12 rigid strips $u^i_{\pm\pm}\colon \{z \in \bC \:|\: 0 \leq \im z \leq \pi/2\} \to \bCP^2$, $i \in \{0,1,2\}$ that contribute to the differential:
\begin{align}
u^0_{\pm\pm}(z)
\coloneqq
\Bigl[
\frac{e^z-1}{e^z+1}
: \pm 1
: \pm 1\Bigr],
\quad
u^1_{\pm\pm}(z)
\coloneqq
\Bigl[
1
: \pm \frac{e^z-1}{e^z+1}
: \pm 1\Bigr]
\quad
u^2_{\pm\pm}(z)
\coloneqq
\Bigl[
1
: \pm 1
: \pm \frac{e^z-1}{e^z+1}\Bigr].
\end{align}
Then $u_{\pm\pm}^0$ goes from $p_{\mp\mp}$ to $p_{\pm\pm}$, $u_{\pm\pm}^1$ goes from $p_{\mp\pm}$ to $p_{\pm\pm}$, and $u_{\pm\pm}^2$ goes from $p_{\pm\mp}$ to $p_{\pm\pm}$, which shows that the left half of the figure in the lemma statement indeed depicts the generators and differential of $CF(\gamma,S^1_\Cl)$.

Next, we consider $CF(\gamma,S^1_\Cl)$.
Using the presentation of $\gamma$ we gave at the beginning of the current section, we can write the generators of $CF(\gamma,S^1_\Cl)$ in the same way as those of $CF(\bRP^2,T^2_\Cl)$, as $p_{\pm\pm} = [1:\pm1:\pm1]$, where we can think of the second sign as differentiating between the two sheets of $\gamma$.
Again using \cite[Prop.~3.1]{alston}, there are 8 rigid strips $u_{\pm\pm}^i\colon \{z \in \bC \:|\: 0 \leq \im z \leq \pi/2\} \to \bCP^1$, $i \in \{0,1\}$ that contribute to the differential:
\begin{align}
v^0_{\pm\pm}(z)
\coloneqq
\Bigl[
\frac{e^z-1}{e^z+1}
: \pm 1\Bigr],
\quad
v^1_{\pm\pm}(z)
\coloneqq
\Bigl[
1
: \pm \frac{e^z-1}{e^z+1}\Bigr],
\end{align}
where the second sign in the subscript of $v_{\pm\pm}^i$ indicates the sheet of $\gamma$ that contains the input generator.
Then $v_{\pm\pm}^0$ goes from $p_{\mp\mp}$ to $p_{\mp\pm}$, and $v_{\pm\pm}^1$ goes from $p_{\mp\pm}$ to $p_{\pm\pm}$.
This justifies the right half of the figure in the lemma statement.
\end{proof}

Denoting the instances of $\mu^1$ on the left resp.\ right by $\mu^1_{\bCP^1}$ resp.\ $\mu^1_{\bCP^2}$, one can immediately see that $(\mu^1_{\bCP^1})^2 = 0$ but $(\mu^1_{\bCP^2})^2 = \id$.
This means that in this case, the isomorphism \eqref{eq:WW_iso} must be obstructed by figure eight bubbling.
The moduli space $\ol\cM$ of quilted strips that relates $CF(\gamma,S^1_\Cl)$ and $CF(\bRP^2,T^2_\Cl)$ is defined like so:

\medskip

\noindent
{\bf Moduli problem:} $\cM$ is the moduli space of holomorphic quilted strips $\ul u = (u_1, u_2)$ of Maslov index 1 satisfying
\begin{gather}
\label{eq:master_eq}
u_2\colon \{z \in \bC \:|\: 0 \leq \im z \leq h\} \to \bCP^2,
\qquad
u_1\colon \{z \in \bC \:|\: h \leq \im z \leq \pi/2\} \to \bCP^1,
\\
u_2(x) \in \bRP^2,
\quad
u_1(x+i\pi/2) \in S^1_\Cl,
\quad
(u_2(x+ih),u_1(x+ih)) \in \Lambda_{S^1}
\qquad
\forall\: x\in \bR
\nonumber
\end{gather}
for some $h \in [0,\pi/2]$.
$\ol\cM$ is the Gromov compactification of $\cM$, as described in \S4, \cite{bw}.

\medskip

Before we present formulas for the quilts making up $\ol\cM$, we explain the idea behind our construction, which applies more generally to constructions of pseudoholomorphic quilts with seam condition coming from symplectic reduction.
$\cM$ is composed of strips pictured in (a) below; fix such a strip $\ul u = (u_1,u_2)$.
Denote by $\pi\colon \bCP^2 \dashrightarrow \bCP^1$ the map sending $[X:Y:Z]$ to $[X:Y]$, which has the properties that (1) $\pi$ is holomorphic where is it defined, (2) $(\id_{\bCP^1} \times \pi)(\Lambda_{S^1}) = \Delta_{\bCP^1}$, and (3) $\pi(\bRP^2) = \bRP^1$, and in particular is Lagrangian.
It follows that if we form the quilt $\wt u = (u_1,\pi\circ u_2)$ as in (b) below, we may ``erase'' the seam and obtain an unquilted strip as in (c), mapping to $\bCP^1$ and with boundary conditions in $S^1_\Cl$ and $\bRP^1$.
(Since $\pi$ is undefined at $[0:0:1]$, $\pi\circ u_2$ may be initially undefined at a discrete subset of the domain, but these singularities are removable.)

\begin{figure}[H]
\centering
\def\svgwidth{1.0\columnwidth}
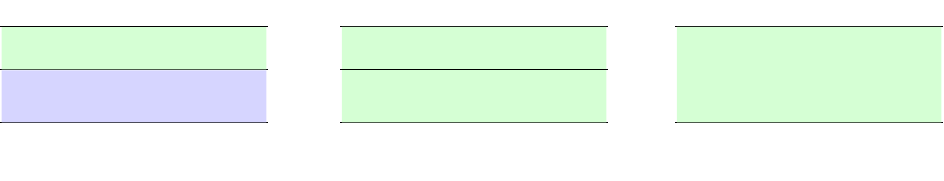
\caption{\label{fig:project}}
\end{figure}

\noindent Holomorphic strips as in (c) are completely characterized in terms of Blaschke products by work of Alston \cite{alston}.
Recovering the original quilt $\ul u$ is then a matter of producing a function $f$ so that $u_2 = [u_2^0:u_2^1:f]$, where we have denoted $\pi\circ u_2 \eqqcolon [u_2^0:u_2^1]$, for homogeneous coordinate functions $u_2^0, u_2^1$ chosen without loss of generality to be real-valued on $\bR$.
Such an $f$ is to satisfy the following conditions: (1) $f$ is meromorphic, (2) $f(x) \in \bR$ for $x$ in the bottom boundary, and (3) $2|f(x+ih)|^2 = |u_2^0(x+ih)|^2 + |u_2^1(x+ih)|^2$ for $x+ih$ in the top boundary.

We now turn to our explicit characterization of $\ol\cM$.
We need one more preparatory lemma, in which we characterize in the setting of $\ol\cM$ the function we denoted above by $f$.

\begin{lemma}[paraphrased from \cite{eremenko}]
\label{lem:eremenko}
Fix $h \in (0,\pi/2)$.
There are two zero-free holomorphic functions $f\colon \{z \in \bC \:|\: 0 \leq \im z \leq h\} \to \bC$ satisfying the following conditions:
\begin{align}
\label{eq:f_pm_conds}
f(x) \in \bR
\:\forall\:
x \in \bR,
\qquad
|f(x+ih)|^2 = \exp(2x) + 1
\:\forall\:
x \in \bR,
\qquad
\lim_{\re z \to \pm\infty}
\left|\frac {f(z)^2}{\exp(2z)+1}\right|
=
1.
\end{align}
They are given by the formula
\begin{align}
\label{eq:f_pm_formula}
f_\pm(z)
=
\pm\exp\left(
\frac1{2\pi i}
\int_{-\infty}^\infty
\frac{1+it\exp(\tfrac\pi{2h}z)}{t-i\exp(\tfrac\pi{2h}z)}
\cdot
\frac{\log\bigl(1+|t|^{4h/\pi}\bigr)}{1+t^2} \,\d t
\right).
\end{align}
Moreover, if we remove the zero-free condition, then the satisfactory functions are exactly those of the form $\prod_{i=1}^k \bigl(\exp(\tfrac{\pi}{2h}z) - \alpha_i\bigr)\big/\bigl(\exp(\tfrac\pi{2h}z) - \ol{\alpha_i}\bigr) \cdot f_\pm(z)$ for some $(\alpha_i) \subset \{-h < \im z < h\}$ invariant as a set under the automorphism $z \mapsto \ol z$, i.e.\ multiples of $f_\pm$ by a suitable Blaschke product.
\end{lemma}

\begin{proof}
First, we address the claim that the functions $f_\pm$ defined in \eqref{eq:f_pm_formula} are zero-free holomorphic functions satisfying \eqref{eq:f_pm_conds}.
By Schwarz reflection, we can equivalently search for holomorphic zero-free functions $f\colon \{-h \leq \im z \leq h\} \to \bC$ satisfying the conditions
\begin{align}
f(x) \in \bR
\:\forall\:
x \in \bR,
\qquad
|f(x\pm ih)|^2 = \exp(2x) + 1
\:\forall\:
x \in \bR,
\qquad
\lim_{\re z \to \pm\infty}
\left|\frac {f(z)^2}{\exp(2z) + 1}\right|
=
1.
\end{align}
The map $z \mapsto \log\bigl((-iz)^{2h/\pi}\bigr)$ is a biholomorphism from $\bH\setminus\{0\}$ to $\{-h\leq\im z\leq h\}$ that sends $\bR_{>0}$ resp.\ $i\bR_{>0}$ resp.\ $\bR_{<0}$ to $\bR - ih$ resp.\ $\bR$ resp.\ $\bR + ih$, so we can equivalently search for holomorphic zero-free functions $g\colon \bH\setminus\{0\} \to \bC$ satisfying
\begin{align}
\label{eq:g_conds}
g(x) \in \bR
\:\forall\:
x \in i\bR_{>0},
\qquad
|g(x)| = \sqrt{|x|^{4h/\pi} + 1}
\:\forall\:
x \in \bR,
\qquad
\lim_{z \to 0, \infty}
\left|\frac {g(z)^2}{z^{4h/\pi} + 1}\right|
=
1.
\end{align}
It follows from \cite[Thm.~1, \S1.2]{de_branges} that the following functions $g_\pm \colon \inte\:\bH \to \bC$ are holomorphic, zero-free, extend continuously to $\bH$, and satisfy the middle equality in \eqref{eq:g_conds}:
\begin{align}
\label{eq:g_pm_formula}
g_\pm(z)
\coloneqq
\pm\exp\left(\frac1{2\pi i}\int_{-\infty}^\infty
\frac{1+tz}{t-z}
\cdot
\frac{\log\left(1+|t|^{4h/\pi}\right)}{1+t^2} \,\dt\right)
\end{align}
In the following bullets, we check that $g_\pm$ extend smoothly to $\bH\setminus\{0\}$ and that these functions satisfy the first and third conditions in \eqref{eq:g_conds}.
\begin{itemize}
\item
To check the first condition in \eqref{eq:g_conds}, we check that the argument of the exponential in \eqref{eq:g_pm_formula} is real for $z = iy$, $y > 0$:
\begin{align}
\frac1{2\pi i}
\int_{-\infty}^0
\frac{1+ity}{t-iy}
\cdot
\frac{\log\left(1+|t|^{4h/\pi}\right)}{1+t^2}
\,\dt
&=
-\frac1{2\pi i}
\int_0^\infty
\frac{1-it\wt y}{t+i\wt y}
\cdot
\frac{\log\left(1+|t|^{4h/\pi}\right)}{1+t^2}
\,\dt
\\
&=
\ol{\frac1{2\pi i}
\int_0^\infty
\frac{1+ity}{t-iy}
\cdot
\frac{\log\left(1+|t|^{4h/\pi}\right)}{1+t^2}
\,\dt}.
\nonumber
\end{align}

\medskip

\item
Next, we must check the third condition in \eqref{eq:g_conds}.
The equality $\lim_{z \to 0} |g_\pm(z)| = 1$ follows from the continuity of $g_\pm$ on $\bH$ and the second condition in \eqref{eq:g_conds}.
The equality $\lim_{z \to \infty} \bigl|g_\pm(z)^2/z^{4h/\pi}\bigr| = 1$ follows from an argument analogous to the one made in \cite[Thm.~2, \S1.2]{de_branges}.

\medskip

\item
Finally, we must show that $g_\pm$ extend smoothly to $\bH\setminus\{0\}$.
It suffices by the Sobolev embedding theorem to show that for every $u \in \bR\setminus\{0\}$, $g_\pm$ is in $W^{2,4}$ on the intersection of a small ball centered at $u$ with $\bH\setminus\{0\}$.
Denoting
\begin{align}
I(z)
\coloneqq
\int_{-\infty}^\infty
\frac{1+tz}{t-z}
\cdot
\frac{\log\bigl(1+|t|^{4h/\pi}\bigr)}{1+t^2}
\,\dt,
\end{align}
this follows from the identities
\begin{align}
I'(z)
=
\int_{-\infty}^\infty
\frac{\log\bigl(1+|t|^{4h/\pi}\bigr)}{(t-z)^2}
\,\dt,
\qquad
I''(z)
=
\int_{-\infty}^\infty
\frac{2\log\bigl(1+|t|^{4h/\pi}\bigr)}{(t-z)^3}
\,\dt
\end{align}
and a simple argument using integration by parts and the differentiability of $\log\bigl(1+|t|^{4h/\pi}\bigr)$ away from $t=0$.
\end{itemize}
Precomposing $g_\pm$ with the biholomorphism from $\{-h\leq\im z\leq h\}$ to $\bH\setminus\{0\}$ that sends $z$ to $i\exp(\tfrac\pi{2h}z)$, we see that the functions $f_\pm$ defined in \eqref{eq:f_pm_formula} do indeed satisfy the conditions in \eqref{eq:f_pm_conds}.

Second, we must show that the holomorphic functions $f\colon \{0 \leq \im z \leq h\} \to \bC$ satisfying \eqref{eq:f_pm_conds} are exactly those functions of the form $\prod_{i=1}^k \bigl(\exp(\pi z/2h) - \alpha_i\bigr)\big/\bigl(\exp(\pi z/2h) - \ol{\alpha_i}\bigr) \cdot f_\pm(z)$ for some $(\alpha_i) \subset \bR_{>0}$.
If $f$ is such a function, then $a \coloneqq f/f_+$ is a holomorphic function on $\{0\leq \im z \leq h\}$ which satisfies the conditions
\begin{align}
a(x) \in \bR
\:\forall\: x \in \bR,
\qquad
|a(x + ih)| = 1
\:\forall\: x \in \bR,
\qquad
\lim_{\re z \to \pm\infty}
|a(z)|
= 1.
\end{align}
Using Schwarz reflection across $\bR$ and then considering the biholomorphism $z \mapsto \log\bigl((-iz)^{2h/\pi}\bigr)$ from $\bH\setminus\{0\}$ to $\{-h\leq\im z\leq h\}$, we see that we may equivalently characterize those holomorphic functions $b \colon \bH\setminus\{0\} \to \bC$ satisfying
\begin{align}
\label{eq:b_conds}
b(x) \in \bR
\:\forall\: x \in i\bR_{>0},
\qquad
|b(x)| = 1
\:\forall\:
x \in \bR\setminus\{0\},
\qquad
\lim_{z \to 0,\infty}
|b(z)|
= 1.
\end{align}
All such functions are of the form
\begin{align}
\label{eq:b_blaschke}
b(z)
=
\prod_{i=1}^k \frac {z-\beta_i}{z-\ol{\beta_i}},
\quad
(\beta_i) \subset \inte \: \bH;
\end{align}
conversely, any function of this form satisfies all the desired conditions in \eqref{eq:b_conds} except possibly the first.
The manipulation
\begin{align}
\prod_{i=1}^k \frac {iy-\beta_i}{iy-\ol{\beta_i}}
=
\prod_{i=1}^k \frac {y^2+\beta_i^2}{y^2-2i\im \beta_i\: y + |\beta_i|^2}
\end{align}
shows that the satisfactory functions are those of the form \eqref{eq:b_blaschke} such that $\{\beta_i\}$ is mapped onto itself by the automorphism of $\inte\:\bH$ that sends $z$ to $-\ol z$.
\end{proof}

\begin{proposition}
$\ol\cM$ consists of 12 components, each of which fibers trivially over the $h$-interval $[0,\pi/2]$.
$\ol\cM_{h=\pi/2}$ consists of 12 Floer strips, which are exactly the contributions to $\mu^1_{\bCP^2}$ shown on the right of the figure in Lem.~\ref{lem:floer_diff}.
$\ol\cM_{h=0}$ consists of 8 Floer strips, which are the contributions to $\mu^1_{\bCP^1}$ shown on the left of the figure in Lem.~\ref{lem:floer_diff}, as well as 4 constant strips with figure eight bubbles attached.
\end{proposition}

\begin{proof}
The quilted strips come in two types, corresponding to whether the underlying strip with target $\bCP^1$ as in Fig.~\ref{fig:project} is constant or is nonconstant and has Maslov index $1$.
We deal with these two cases in the following two bullets.
\begin{itemize}
\item {\bf Quilted strips with constant projection.}
Fix $h \in (0,\pi/2)$.
In this case, we must have $u_1(z) = [1:\pm1]$.
To compute $u_2$, we must characterize the minimal-energy nonconstant holomorphic functions $f\colon \{0\leq \im z \leq h\} \to \bC$ satisfying the conditions
\begin{align}
f(x) \in \bR
\:\forall\:
x \in \bR,
\qquad
|f(x+ih)| = 1
\:\forall\:
x \in \bR,
\qquad
\lim_{\re z \to \pm\infty}
|f(z)|
=
1.
\end{align}
These are exactly the functions $\pm\bigl(\exp(\tfrac\pi{2h}z)-1\bigr)/\bigl(\exp(\tfrac\pi{2h}z)+1\bigr)$, so we see that there are four valid quilted strips with constant projection:
\begin{align}
u_2(z) = \left[1:\pm1:\pm\frac{\exp(\tfrac \pi{2h}z)-1}{\exp(\tfrac\pi{2h}z)+1}\right],
\qquad
u_1(z) = [1:\pm1],
\end{align}
where the sign in $u_1$ is the same as the first sign in $u_2$.

In the $h \to \pi/2$ limit we get the (unquilted) strips
\begin{align}
\wt u_2(z) = \left[1:\pm1:\pm\frac{\exp(z)-1}{\exp(z)+1}\right].
\end{align}
In the $h \to 0$ limit we get the constant strips $\wt u_1(z) = [1:\pm 1]$, and a figure eight bubbles off at the origin, consisting of the following maps:
\begin{gather}
v_2\colon \{z \:|\: 0 \leq \im z \leq \pi/2\} \to \bCP^2,
\qquad
v_1\colon \{z \:|\: \pi/2 \leq \im z\} \to \bCP^1,
\\
v_2(z) = \left[1:\pm1:\pm\frac{\exp(z)-1}{\exp(z)+1}\right],
\qquad
v_1(z) = [1:\pm1].
\nonumber
\end{gather}
Note that this is a sheet-switching eight in the sense of \cite[Rmk.\ 2.5]{bw}.

\medskip

\item {\bf Quilted strips with Maslov-$1$ projection.}
Fix $h \in (0,\pi/2)$.
In this case, we must have $u_1(z) = [\exp(z)+1:\pm(\exp(z)-1)]$ or $u_1(z) = [\exp(z)-1:\pm(\exp(z)+1)]$.
To compute $u_2$, we must characterize the zero-free holomorphic functions $f \colon \{0 \leq \im z \leq h\} \to \bC$ satisfying the conditions
\begin{align}
f(x) \in \bR
\:\forall\:
x \in \bR,
\qquad
|f(x+ih)|^2
=
\exp(2x) + 1
\:\forall\:
x \in \bR,
\qquad
\lim_{\re z \to \pm\infty}
\left|\frac{f(z)^2}{\exp(2z)+1}\right|
= 1.
\end{align}
As we showed in Lemma~\ref{lem:eremenko}, there are exactly two functions $f_\pm$ that satisfy these conditions, given by the formula \eqref{eq:f_pm_formula}.
There are therefore eight valid quilted strips with Maslov-1 projection:
\begin{align}
&u_2(z)
=
\left[
\exp(z)+1
:
\pm(\exp(z)-1)
:
f_\pm(z)
\right],
\qquad
u_1(z)
=
[\exp(z)+1:\pm(\exp(z)-1)],
\\
&u_2(z)
=
\left[
\exp(z)-1
:
\pm(\exp(z)+1)
:
f_\pm(z)
\right],
\qquad
u_1(z)
=
[\exp(z)-1:\pm(\exp(z)+1)],
\nonumber
\end{align}
where in each line the sign in $u_1$ is the same as the first sign in $u_2$.
(Note that if we had used an $f$ which was not zero-free, then the resulting quilted strip would have had quilted Maslov index greater than 1.)

In the $h \to \pi/2$ limit we get the strips
\begin{align}
&\wt u_2(z) = \left[\exp(z)+1:\pm(\exp(z)-1):\pm(\exp(z)+1)\right],
\\
&\wt u_2(z) = \left[\exp(z)-1:\pm(\exp(z)+1):\pm(\exp(z)+1)\right]
\nonumber
\end{align}
while in the $h \to 0$ limit we get the strips
\begin{gather}
\wt u_1(z) = \left[\exp(z)+1:\pm(\exp(z)-1)\right],
\qquad
\wt u_1(z) = \left[\exp(z)-1:\pm(\exp(z)+1)\right].
\end{gather}
\end{itemize}
\end{proof}

\section{A figure eight bubble predicted by Akveld--Cannas~da~Silva--Wehrheim}
\label{sec:question}

Katrin Wehrheim described the following question to the author in a 2015 private communication \cite{acw}:

\begin{question}[Akveld--Cannas~da~Silva--Wehrheim, private communication]
\label{q:acw}
Can one explicitly produce a quilted disk $\ul u = (u_1,u_2)$ consisting of holomorphic maps
\begin{align}
u_1\colon \{z \:|\: \im z\geq 1\} \to \bCP^1,
\qquad
u_2\colon \{z \:|\: 0 \leq \im z\leq 1\} \to \bCP^2
\end{align}
satisfying seam and boundary conditions in $\Lambda_{S^1}$ resp.\ $L_\AC \subset \bCP^2$,
\begin{gather}
\label{eq:L_AC}
\bigl(u_1(x+i),u_2(x+i)\bigr) \in \Lambda_{S^1} \:\forall\: x \in \bR,
\qquad
u_2(x) \in L_\AC \:\forall\: x \in \bR,
\\
L_\AC \coloneqq \bigl\{[A+iB:i\sqrt 2C:A-iB] \:|\: A,B,C\in\bR,\:|A|^2+|B|^2+|C|^2=1\bigr\},
\nonumber
\end{gather}
and having symplectic area one-fourth that of $\bCP^1$ (hence quilted Maslov index 1)?
\end{question}

\noindent
Abstractly, Akveld--Cannas~da~Silva--Wehrheim showed that such a figure eight bubble must exist: otherwise, $CF(\bRP^1,L_\AC\circ \Lambda_{S^1}^T = S^1_\Cl)$ and $CF(\bRP^1\circ\Lambda_{S^1} = \bRP^1\times S^1_\Cl,L_\AC)$ would be isomorphic chain complexes, which they are not, since the latter is not even a chain complex.
(This follows from the facts that $\bRP^1 \times S^1_\Cl$ is Hamiltonian-isotopic to $T^2_\Cl$, and that by Rmk.~3.6 of \cite{cannas}, $L_\AC$ is Hamiltonian-isotopic to the standard $\bRP^2 \subset \bCP^2$.)

In this section, we answer Question~\ref{q:acw} in the affirmative:

\begin{proposition}
The following two quilts satisfy the conditions of Question~\ref{q:acw}:
\begin{align}
u_1(z) \coloneqq [ \pm(z+6i) : iz ],
\qquad
u_2(z) \coloneqq [ \pm(z+6i) : iz : \pm(z-6i) ],
\end{align}
where the three signs appearing in these formulas are all the same.
\end{proposition}

\begin{proof}
%
First, we check the seam condition.
The only thing we need to check is the condition $2|Z|^2 = |X|^2+|Y|^2$ on $\bR + i$:
\begin{align}
2|(x+i)-6i|^2
&=
2(x^2 + 25)
=
(x^2+49) + (x^2+1)
=
|(x+i)+6i|^2 + |i(x+i)|^2.
\end{align}
Second, we check the boundary condition.
The zeroth and second components of $u_2(x) = [\pm(x+6i) : ix : \pm(x-6i)]$ are conjugates, and the first component is imaginary; $u_2(x)$ therefore lies in $L_\AC$.
Finally, we calculate the symplectic area.
Recall that the multiple $\om_{\bCP^n}$ of the Fubini--Study form having monotonicity constant 1 is given by the following formula:
\begin{gather}
\omega_{\bCP^n}
=
\frac {(n+1)i} {2\pi(1+|\ul z|^2)^2}\left(
\sum_{1\leq k\leq n} (1 + |\ul z|^2 - |z_k|^2) \d z_k\wedge\d\ol{z_k}
- \sum_{j \neq k} \ol{z_j}z_k\d z_j\wedge\d\ol{z_k}
\right).
\end{gather}
We now use these expressions for $\om_{\bCP^1}, \om_{\bCP^2}$ to compute the symplectic areas of $u_1$ and $u_2$:
\begin{gather}
\int_{y\geq 1} u_1^*\omega_{\bCP^1}
=
\int_{y\geq 1}
\Bigl(z \mapsto \pm\frac{z+6i}{iz}\Bigr)^*
\omega_{\bCP^1}
=
\int_1^\infty \int_{-\infty}^\infty \frac{18}{\pi(x^2+y^2+6y+18)^2}\d x\d y
= \frac 1 5,
\\
\int_{0\leq y\leq 1} u_2^*\omega_{\bCP^2}
=
\int_{0\leq y\leq1}
\Bigl(z \mapsto \pm\Bigl(\frac{z+6i}{iz},\frac{z-6i}{iz}\Bigr)\Bigr)^*
\om_{\bCP^2}
=
\int_0^1\int_{-\infty}^\infty \frac{72}{\pi(x^2+y^2+24)^2}\d x\d y
= \frac 3 {10}.
\nonumber
\end{gather}
The total symplectic area is therefore $1/2$, which is indeed one-fourth the area of $\bCP^1$: the area of $\bCP^1$ is equal to $c_1(T\bCP^1)$ evaluated on the fundamental class, which is 2.
%
%
%
\end{proof}

To conclude this example, we illustrate the projections of these quilts to the moment triangle of $\bCP^2$.
Recall that the moment map of $\bCP^2$ is
\begin{gather}
\mu_{\bCP^2}\colon \bCP^2 \to \bR^2,
\qquad
\mu([X:Y:Z])
\coloneqq
\left(-\frac 1 2\frac {|Y|^2}{|X|^2+|Y|^2+|Z|^2},
-\frac 1 2\frac {|Z|^2}{|X|^2+|Y|^2+|Z|^2}\right),
\nonumber
\end{gather}
and that the image of $\mu_{\bCP^2}$ is the triangle with vertices $(0,0)$, $(-\tfrac12,0)$, and $(0,-\tfrac12)$.
We now compute the images of all the objects involved in this construction.
\begin{itemize}
\item First, we compute the projections of $\Lambda_{S^1}$ and $L_\AC$ to the moment triangle:
\begin{gather}
\mu_{\bCP^2}(\Lambda_{S^1})
=
\left\{
\left(-\frac13\frac{|Y|^2}{|X|^2+|Y|^2},-\frac16\right)\right\}
=
\text{line segment from } (-\tfrac13,-\tfrac16) \text{ to } (0,-\tfrac16),
\\
\mu_{\bCP^2}(L_\AC)
=
\left\{\left(-\frac12\frac{C^2}{A^2+B^2+C^2},-\frac14\left(1-\frac{C^2}{A^2+B^2+C^2}\right)\right)\right\}
=
\text{segment from } (-\tfrac12,0) \text{ to } (0,-\tfrac14).
\nonumber
\end{gather}

\item Next, we note that we can view the moment interval of $\bCP^1 = \mu_{\bCP^2}^{-1}(-\tfrac16)/S^1$ as a subset of the moment triangle of $\bCP^2$, by lifting $\mu_{\bCP^1}\colon \bCP^1 \to \bR$ to $\wt\mu_{\bCP^1}\colon \bCP^1 \to \bR^2$.
Indeed, $\mu_{\bCP^2}^{-1}(-\tfrac16)$ is characterized by the equation $2|Z|^2=|X|^2+|Y|^2$, so we can lift $\mu_{\bCP^1}$ to $\wt\mu_{\bCP^1}$ like so:
\begin{gather}
\wt\mu_{\bCP^1}([X:Y])
=
\left(-\frac 1 2\frac {|Y|^2}{|X|^2+|Y|^2+|Z|^2},
-\frac 1 2\frac {|Z|^2}{|X|^2+|Y|^2+|Z|^2}\right)
=
\left(-\frac13\frac{|Y|^2}{|X|^2+|Y|^2},-\frac16\right).
\end{gather}

\item Finally, we compute the projections of the images of $u_1$ and $u_2$:
\begin{gather}
\wt\mu_{\bCP^1}(u_1(z))
=
\left(-\frac16\frac{|z|^2}{|z+3i|^2+9},
-\frac16\right),
\qquad
\mu_{\bCP^2}(u_2(z))
=
\left(-\frac16\frac{|z|^2}{|z|^2+24},
-\frac16\frac{|z-6i|^2}{|z|^2+24}\right).
\end{gather}
The right boundary of the projection of the image of $u_2$ is the projection of $\{u_2(iy)\:|\:0\leq y\leq1\}$, which is the segment of the ellipse $Z(100x^2+16xy+16y^2+20x+8y+1)$ between $(0,-\tfrac14)$ and $(-\tfrac1{150},-\tfrac16)$.
Note that the projections of the image of $u_1$ and $u_2$ are doubly-covered: $(x,y)$ and $(-x,y)$ get sent to the same point.
\end{itemize}

In the figure below we show these projections: those of $\Lambda_{S^1}$, $L_\AC$, and the images of $u_1$ and $u_2$.

\begin{figure}[H]
\centering
\def\svgwidth{0.8\columnwidth}
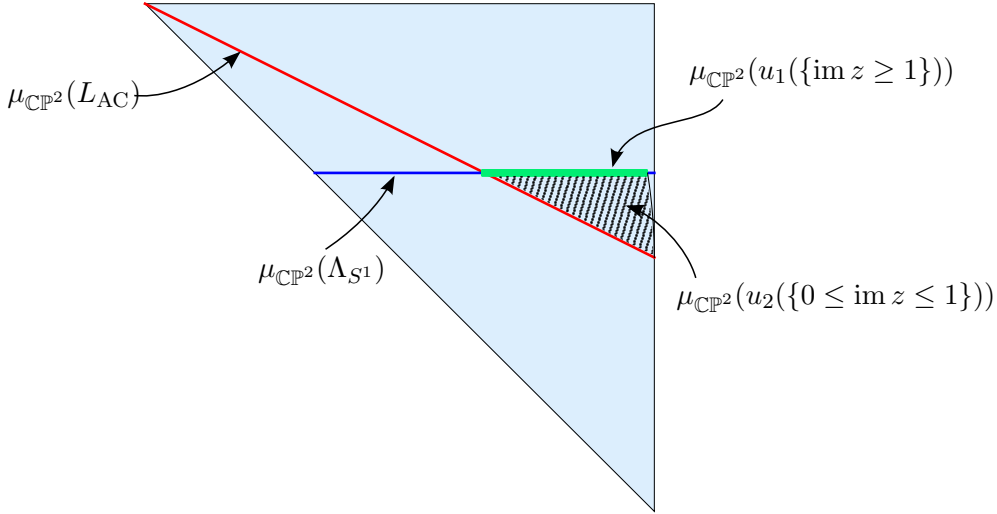
\caption{The large, light-blue triangle is the moment polytope of $\bCP^2$.
The diagonal red line segment is the projection of $L_\AC$.
The horizontal blue line segment (which extends all the way to the right side of the moment triangle) is the projection of $\Lambda_{S^1}$, so we can think of $\bCP^1$ as the preimage of the blue line quotiented by the action of $S^1$ that rotates the last coordinate.
The shaded region is the projection of the image of $u_2$, and the green line segment is the projection of the image of $u_1$.
Note that the shaded region only intersects the right edge at a single point, corresponding to $u_2(0) = [1:0:-1]$.}
\end{figure}

\end{document}